\documentclass[12pt,a4paper]{amsart}
\usepackage{xcolor}
\usepackage{mathtools}
\usepackage{amsmath,amstext,amssymb,amsthm,amsfonts}
\usepackage[all]{xy}
\newcommand{\nc}{\newcommand}
\renewcommand{\top}{\mathrm{top}}
\nc{\one}{\mbox{\bf 1}}
\nc{\invtensor}{\underset{\leftarrow}{\otimes}}
\nc{\const}{\operatorname{const}}

\nc{\ad}{\operatorname{ad}}
\nc{\tr}{\operatorname{tr}}
\nc{\tp}{\operatorname{top}}
\nc{\rank}{\operatorname{rank}}
\nc{\corank}{\operatorname{corank}}
\nc{\codim}{\operatorname{codim}}
\nc{\sdim}{\operatorname{sdim}}
\nc{\mult}{\operatorname{mult}}

\nc{\spn}{\operatorname{span}}
\nc{\Sym}{\operatorname{Sym}}
\nc{\sym}{\operatorname{sym}}
\nc{\id}{\operatorname{id}}
\nc{\Id}{\operatorname{Id}}
\nc{\Ree}{\operatorname{Re}}

\nc{\htt}{\operatorname{ht}}
\nc{\sch}{\operatorname{sch}}
\nc{\str}{\operatorname{str}}

\nc{\Ker}{\operatorname{Ker}}
\nc{\rker}{\operatorname{rKer}}
\nc{\im}{\operatorname{Im}}
\nc{\osp}{\mathfrak{osp}}
\nc{\sgn}{\operatorname{sgn}}
\nc{\F}{\operatorname{F}}
\nc{\Mod}{\operatorname{Mod}}
\nc{\Mat}{\operatorname{Mat}}
\nc{\Soc}{\operatorname{Soc}}
\nc{\Inj}{\operatorname{Inj}}
\nc{\Hom}{\operatorname{Hom}}
\nc{\End}{\operatorname{End}}
\nc{\supp}{\operatorname{supp}}
\nc{\Card}{\operatorname{Card}}
\nc{\Ann}{\operatorname{Ann}}
\nc{\Ind}{\operatorname{Ind}}
\nc{\Coind}{\operatorname{Coind}}
\nc{\wt}{\operatorname{wt}}
\nc{\ch}{\operatorname{ch}}
\nc{\Stab}{\operatorname{Stab}}
\nc{\Sch}{{\mathcal S}\mbox{\em ch}}
\nc{\Irr}{\operatorname{Irr}}
\nc{\Spec}{\operatorname{Spec}}
\nc{\Prim}{\operatorname{Prim}}
\nc{\Aut}{\operatorname{Aut}}
\nc{\Ext}{\operatorname{Ext}}

\nc{\Fract}{\operatorname{Fract}}
\nc{\gr}{\operatorname{gr}}
\nc{\deff}{\operatorname{def}}
\nc{\HC}{\operatorname{HC}}

\nc{\red}{\operatorname{red}}
\nc{\wdchi}{\widetilde{\chi}}
\nc{\wdH}{\widetilde{H}}
\nc{\wdN}{\widetilde{N}}
\nc{\wdM}{\widetilde{M}}
\nc{\wdO}{\widetilde{O}}
\nc{\wdR}{\widetilde{R}}
\nc{\wdS}{\widetilde{S}}
\nc{\wdV}{\widetilde{V}}

\nc{\wdC}{\widetilde{C}}

\nc{\Obj}{\operatorname{Obj}}
\nc{\Dglie}{\operatorname{{\mathcal D}glie}}
\nc{\Fin}{\operatorname{{\mathcal F}in}}

\nc{\Adm}{\operatorname{\mathcal{A}dm}}
\nc{\Sg}{{\cS(\fg)}}
\nc{\Shg}{{\cS(\fhg)}}
\nc{\Ug}{{\cU(\fg)}}
\nc{\Uhg}{{\cU(\fhg)}}
\nc{\Sh}{{\cS(\fh)}}
\nc{\Uh}{{\cU(\fh)}}
\nc{\Uhh}{{\cU(\fhh)}}
\nc{\Zg}{{{\mathcal{Z}}(\fg)}}

\nc{\Vir}{{\mathcal{V}ir}}
\nc{\NS}{{\mathcal{N}S}}

\nc{\tZg}{{\widetilde{\mathcal Z}({\mathfrak g})}}
\nc{\Zk}{{\mathcal Z}({\mathfrak k})}

\newcommand{\D}{\mathcal{D}}

\nc{\Up}{{\mathcal U}({\mathfrak p})}
\nc{\Ah}{{\mathcal A}({\mathfrak h})}
\nc{\Ag}{{\mathcal A}({\mathfrak g})}
\nc{\Ap}{{\mathcal A}({\mathfrak p})}
\nc{\Zp}{{\mathcal Z}({\mathfrak p})}
\nc{\cR}{\mathcal R}
\nc{\cS}{\mathcal S}
\nc{\cT}{\mathcal{T}}
\nc{\cY}{\mathcal Y}
\nc{\cB}{\mathcal B}
\nc{\cU}{\mathcal U}
\nc{\cH}{\mathcal H}
\nc{\cM}{\mathcal M}
\nc{\cL}{\mathcal L}
\nc{\cF}{\mathcal F}
\nc{\fg}{\mathfrak g}

\nc{\fo}{\mathfrak o}

\nc{\CO}{\mathcal O}
\nc{\CR}{\mathcal R}
\nc{\Cl}{\mathcal {C}\ell}

\nc{\cW}{\mathcal{W}}
\nc{\bM}{\mathbf{M}}
\nc{\bL}{\mathbf{L}}
\nc{\bN}{\mathbf{N}}

\nc{\zq}{\mathpzc q}

\nc{\fl}{\mathfrak l}
\nc{\fn}{\mathfrak n}
\nc{\fm}{\mathfrak m}
\nc{\fp}{\mathfrak p}
\nc{\fh}{\mathfrak h}
\nc{\ft}{\mathfrak t}
\nc{\fk}{\mathfrak k}
\nc{\fb}{\mathfrak b}
\nc{\fs}{\mathfrak s}

\nc{\psl}{\mathfrak{psl}}
\nc{\fB}{\mathfrak B}

\nc{\vareps}{\varepsilon}
\nc{\varesp}{\varepsilon}
\nc{\veps}{\varepsilon}

\nc{\fsl}{\mathfrak{sl}}
\nc{\fgl}{\mathfrak{gl}}
\nc{\fso}{\mathfrak{so}}

\nc{\fpq}{\mathfrak{pq}}
\nc{\fq}{\mathfrak q}
\nc{\fsq}{\mathfrak{sq}}
\nc{\fpsq}{\mathfrak{psq}}

\nc{\fhg}{\hat{\fg}}
\nc{\fhn}{\hat{\fn}}
\nc{\fhh}{\hat{\fh}}
\nc{\fhb}{\hat{\fb}}
\nc{\hrho}{\hat{\rho}}

\nc{\hsl}{\hat{\fsl}}
\nc{\fpo}{\mathfrak{po}}
\nc{\dirlim}{\underset{\rightarrow}{\lim}\,}
\nc{\nen}{\newenvironment}
\nc{\ol}{\overline}
\nc{\ul}{\underline}
\nc{\ra}{\rightarrow}
\nc{\lra}{\longrightarrow}
\nc{\Lra}{\Longrightarrow}
\nc{\bo}{\bar{1}}
\nc{\Lla}{\Longleftarrow}

\nc{\Llra}{\Longleftrightarrow}

\nc{\thla}{\twoheadleftarrow}

\nc{\lang}{(}
\nc{\rang}{)}

\nc{\hra}{\hookrightarrow}

\nc{\iso}{\overset{\sim}{\lra}}

\nc{\ssubset}{\underset{\not=}{\subset}}

\nc{\vac}{|0\rangle}

\nc{\Thm}[1]{Theorem~\ref{#1}}
\nc{\Prop}[1]{Proposition~\ref{#1}}
\nc{\Lem}[1]{Lemma~\ref{#1}}
\nc{\Cor}[1]{Corollary~\ref{#1}}
\nc{\Conj}[1]{Conjecture~\ref{#1}}
\nc{\Claim}[1]{Claim~\ref{#1}}
\nc{\Defn}[1]{Definition~\ref{#1}}
\nc{\Exa}[1]{Example~\ref{#1}}
\nc{\Rem}[1]{Remark~\ref{#1}}
\nc{\Note}[1]{Note~\ref{#1}}
\nc{\Quest}[1]{Question~\ref{#1}}
\nc{\Hyp}[1]{Hypoth\`ese~\ref{#1}}
\nen{thm}[1]{\label{#1}{\bf Theorem.\ } \em}{}
\nen{prop}[1]{\label{#1}{\bf Proposition.\ } \em}{}
\nen{lem}[1]{\label{#1}{\bf Lemma.\ } \em}{}
\nen{cor}[1]{\label{#1}{\bf Corollary.\ } \em}{}
\nen{conj}[1]{\label{#1}{\bf Conjecture.\ } \em}{}

\nen{claim}[1]{\label{#1}{\bf Claim.\ } \em}{}

\nen{defn}[1]{\label{#1}{\bf Definition.\ } }{}
\nen{exa}[1]{\label{#1}{\bf Example.\ } }{}
\nen{rem}[1]{\label{#1}{\em Remark.\ } }{}
\nen{note}[1]{\label{#1}{\em Note.\ } }{}
\nen{exer}[1]{\label{#1}{\em Exercise.\ } }{}
\nen{sket}[1]{\label{#1}{\em Sketch of proof.\ } }{}
\nen{quest}[1]{\label{#1}{\bf Question.\ } \em}{}

\nen{hyp}[1]{\label{#1}{\bf Hypoth\`ese.\ } \em}{}
\begin{document}

\setcounter{section}{0}
\setcounter{tocdepth}{1}

\title{Integrable modules over affine Lie superalgebras  $\mathfrak{sl}(1|n)^{(1)}$ }

\author{Maria Gorelik,  Vera Serganova }

\address[]{Dept. of Mathematics, The Weizmann Institute of Science,Rehovot 7610001, Israel}
\email{maria.gorelik@weizmann.ac.il}

\address[]{Dept. of Mathematics, University of California at Berkeley ,Berkeley CA 94720}
\email{serganov@math.berkeley.edu}
\thanks{Supported in part by BSF Grant 2012227.}

\begin{abstract}
We describe the category of integrable $\mathfrak{sl}(1|n)^{(1)}$-modules
with the positive central charge and show that the irreducible modules
provide the full set of irreducible  representations for the corresponding simple vertex algebra.
\end{abstract}
\maketitle

\section{Introduction}
Let $\fg$ be the Kac-Moody superalgebra $\fsl(1|n)^{(1)}, n\geq 2$.
Recall that $\fg_{\ol{0}}=\fgl_n^{(1)}$.
We call a $\fg$-module {\em integrable}
if it is integrable over the affine Lie algebra $\fsl_n^{(1)}$,
locally finite over the Cartan subalgebra $\fh\subset \fgl_n^{(1)}$ and with finite-dimensional generalized
$\fh$-weight spaces.

We normalize the invariant bilinear form on $\fg$ in the usual way
($(\alpha,\alpha)=2$ for the non-isotropic roots $\alpha$).
Let $\cF_k$ be the category of the finitely generated  integrable $\fg$-modules
with central charge $k$.
 This category is empty for $k\not\in\mathbb{Z}_{\geq 0}$.
For $k\in \mathbb{Z}_{>0}$, the irreducible objects in $\cF_k$ are
highest weight modules, which were classified in~\cite{KW2} (see~\cite{FR}, Theorem C).
In this paper we study the category $\cF_k$ for $k\in\mathbb{Z}_{>0}$:
we describe the blocks  (see~\Cor{cortyp} and~\Thm{thmatyp}) in
 terms of quivers with relations and
 show that Duflo-Serganova functor
provides an invariant for the atypical blocks (see~\Cor{corFx})
and this invariant separates the blocks.

In Section~\ref{sect6} we study modules over the simple affine vertex superalgebra $V_k(\fg)$. Recall that the modules over affine vertex algebra $V^k(\fg)$ are the restricted  $\fg$-modules of level $k$.

Let
$\ft$ be an affine Lie algebra, $V^k(\ft)$ be the affine vertex algebra with central charge $k$ and $V_k(\ft)$ denote its simple
quotient. Let $k$ be such that
 $V_k(\ft)$ is integrable (as a $\ft$-module); for an appropriate
 normalization of the bilinear form, this means that $k$ is a non-negative integer.
 In this case,
the $V_k(\ft)$-modules  are the restricted integrable $[\ft,\ft]$-modules of level $k$. For $k\geq 0$ the  irreducible restricted integrable $[\ft,\ft]$-modules of level $k$ are  highest weight modules.
In particular, the vertex algebra $V_k(\ft)$ is rational and regular:

(a) there are finitely many (up to isomorphism)
irreducible  $V_k(\ft)$-modules;

(b) any representation is completely reducible.

For positive energy $V_k(\ft)$-modules
these results were  proven in~\cite{FZ}.
In~\cite{DLM} it is shown that any $V_k(\ft)$-module is a direct sum of positive energy modules, which implies the general result.

Let $\fg$ be an (untwisted) affine Lie superalgebra and $\fg_{\ol{1}}\not=0$.
Let $\fg^{\#}$ be the "largest affine subalgebra"
 of $\fg_{\ol{0}}$ (see~\S~\ref{largest}). Let $k$ be such that
 $V_k(\fg)$ is integrable as a $\fg^{\#}$-module (for an appropriate
 normalization of the bilinear form, this means that $k$ is a non-negative integer).
 We prove that the $V_k(\fg)$-modules  are the restricted $\fg$-modules of level $k$
which are $\fg^{\#}$-integrable. We show that
the irreducible positive energy $V_k(\fg)$-modules are highest weight modules
if and only if the Dynkin diagram of $\fg_{\ol{0}}$ is connected ($\fg$
is  $\fsl(1|n)^{(1)}$ or
$\mathfrak{osp}(n|m)^{(1)}$ for $n=1,2$).

 For $\fg=\fsl(1|n)^{(1)}$ one has
 $\fg^{\#}=\fsl_n^{(1)}$. For a positive integer $k$ the category of
 finitely generated $V_k(\fg)$-modules with finite-dimensional generalized
 weight spaces is the full subcategory of $\cF^k$ with the objects annihilated
 by the Casimir operator.

Partial results of this paper were reported at the conferences
in Uppsala in June 2016, Kyoto in October 2016 and Vienna in January 2017.

{\em Acknowledgments.}
We are grateful to V.~Kac  for helpful discussions.

\section{Preliminaries}
Let $\fg=\fsl(1|n)^{(1)}$.
Recall that by definition an integrable $\fg$-module
is integrable over the affine Lie subalgebra $\fsl_n^{(1)}\subset \fg_{\ol{0}}$
and locally finite over the Cartan subalgebra $\fh$. Recall also that
$\fh\cap [\fsl_n^{(1)},\fsl_n^{(1)}]$ acts diagonally on any integrable
$\fsl_n^{(1)}$-module.

Note that $\cF_k$ is the full subcategory in the category
$\mathcal O$. In particular, it is equipped with a covariant duality functor
$\D$ inherited from the contragredient duality in category $\mathcal{O}$.
For any simple object $L$ we have $\D(L)\simeq L$. In particular,
$\Ext^1(L,L')=\Ext^1(L',L)$ for any two simple objects $L$ and $L'$.

\subsection{Roots and sets of simple roots}
View $\fg$ as the affinization of $\dot{\fg}=\fsl(1|n)$. Choose a basis $\varepsilon_1,\cdot,\varepsilon_n$ of
$\dot{\fh}^*$ such that the invariant form is given by
$$(\varepsilon_i,\varepsilon_j)=\begin{cases}-1\text{ if }i\neq j\\0\ \text{ if }i=j\end{cases}.$$
Even roots of $\dot{\fg}$ are of form
$$\{\varepsilon_i-\varepsilon_j,\,|,i\neq j,\,i,j=1,\dots,n\},$$
and all odd roots are isotropic
$$\{\pm\varepsilon_i\,|,i=1,\dots,n\}.$$
By $\delta$ we denote the smallest positive imaginary root of $\fg$. Then all real roots of $\fg$ are of the form
$\alpha+k\delta$, where $k\in\mathbb Z$ and $\alpha$ is a root of $\dot{\fg}$.

For a set of simple roots $\Sigma$ we consider the standard partial order on $\fh^*$ given by
$\lambda\geq_{\Sigma}\mu$ if and only if $\lambda-\mu\in\mathbb{Z}_{\geq 0}\Sigma$.
We denote by $\rho$ the Weyl vector of $\Sigma$ (i.e., $\rho\in\fh^*$ such that
$2(\rho,\alpha)=(\alpha,\alpha)$ for $\alpha\in\Sigma$).

We fix a triangular decomposition of $\fg_{\ol{0}}$ and consider only triangular decompositions of $\fg$ which are compatible with it
(i.e., $\Delta_{\ol{0}}^+$ is fixed). We denote such sets of simple roots by $\Sigma$, $\Sigma'$, etc.
In fact, the category $\mathcal O$ depends only on a triangular decomposition of the even part.

\subsubsection{}\label{uniqueSigma}
We fix a set of simple roots $\Pi_0$ of $\Delta^+_{\ol{0}}$ and let $\mathcal B$ denote the set of all sets of simple roots $\Sigma$ such that
$\Pi_0\subset\mathbb Z_{\geq 0}\Sigma$. To be precise let
$$\Pi_0=\{\varepsilon_1-\varepsilon_2,\dots,\varepsilon_{n-1}-\varepsilon_n,\varepsilon_n-\varepsilon_1+\delta\}.$$
For example,
\begin{equation}\label{sigma}
\Sigma=\{-\varepsilon_1+\delta,\varepsilon_1-\varepsilon_2,\dots,\varepsilon_{n-1}-\varepsilon_n,\varepsilon_n\}\in\mathcal B.
\end{equation}

For any odd root $\beta$ there exists a unique $\alpha\in\Pi_0$ such that $(\alpha,\beta)=-1$
and a unique $\alpha'\in\Pi_0$ such that $(\beta,\alpha')=1$; the set
$$\Sigma=\{\beta,\alpha'-\beta\}\cup(\Pi_0\setminus\{\alpha'\}).$$
is a unique set of simple roots  containing $\beta$.

Note that all $\Sigma\in\mathcal B$ have the same Dynkin diagram. Every $\Sigma$ contains exactly two odd roots $\beta_1$ and $\beta_2$,
$(\beta_1,\beta_2)=1$ and all roots of $\Pi_0$ are even roots of $\Sigma$ and $\beta_1+\beta_2$. The Dynkin diagram is a cycle with $n+1$ nodes: there are two nodes
which correspond to the odd isotropic roots and these nodes are adjacent.
The minimal imaginary positive root $\delta$ is the sum of all simple roots.

\subsubsection{Odd reflections}\label{odd}
Recall that for an odd root $\beta$ belonging to a set of simple roots $\Sigma$, the odd reflection $r_{\beta}$
gives another set of simple roots $r_{\beta}\Sigma$ which contains $-\beta$, the roots $\alpha\in\Sigma\setminus\{\beta\}$,
which are orthogonal to $\beta$, and the roots $\alpha+\beta$ for $\alpha\in\Sigma$ which are not orthogonal to $\beta$.
One has
\begin{equation}\label{DeltaSigma}
\Delta^+(r_{\beta}\Sigma)=(\Delta^+(\Sigma)\setminus\{\beta\})\cup\{-\beta\}.
\end{equation}

Any two sets of simple roots in $\mathcal B$ are connected by a chain of odd reflections.
We call a chain ``proper'' if it does not have loops (i.e. subsequences of the form $r_{\beta}r_{-\beta}$).
Two sets of simple roots are connected by a unique ``proper'' chain of odd reflections.
We call $\Sigma$ and $\Sigma'$ {\it adjacent} if they are obtained from each other by one odd reflection. The adjacency graph with vertices in $\mathcal B$ is an
infinite string (every vertex has two adjacent vertices).
For any two $\Sigma,\Sigma'\in\mathcal B$ there is a unique proper chain of odd reflections connecting them. We denote by $d(\Sigma,\Sigma')$ the number of odd
reflections in this chain.

Let $\Sigma$ be a set of simple roots. One readily sees that the chain
$r_{\beta_s}r_{\beta_{s-1}}\ldots r_{\beta_1}\Sigma$ is proper if and only if
$\beta_1,\ldots,\beta_s\in\Delta^+(\Sigma)$. Let $\beta\not\in\Sigma$ be an odd root and
$\Sigma'$ be a set of simple roots containing $\beta$ (by above, $\Sigma'$ is unique).
If $\beta\in\Delta^+(\Sigma)$, then the proper
chain which connects $\Sigma$ and $\Sigma'$ does not contain the reflections
$r_{\pm \beta}$; if $\beta\in -\Delta^+(\Sigma)$, then the proper chain is of the form
$\Sigma'=r_{\beta_s}r_{\beta_{s-1}}\ldots r_{\beta_1}\Sigma$, where $\beta_s=\beta$.

\subsection{Simple modules}
 For any $\Sigma\in\mathcal B$ we denote by $L_{\Sigma}(\lambda)$ the irreducible module of highest weight $\lambda$ with respect to the Borel
subalgebra corresponding to $\Sigma$.
Given an irreducible highest weight module
 $L$ and $\Sigma\in\mathcal B$ we set $\rho wt_{\Sigma} L:=\lambda$
if $L=L_{\Sigma}(\lambda-\rho_{\Sigma})$ (where $\rho_{\Sigma}$ is the Weyl vector for $\Sigma$). For an odd root
$\alpha\in\Sigma$  one has
\begin{equation}\label{lambdalambda'}
\rho wt_{r_{\alpha}\Sigma} L=\begin{cases}\rho wt_{\Sigma} L \quad\text{if}
\quad (\lambda,\alpha)\neq 0,\\    \rho wt_{\Sigma} L+\alpha\quad\text{if}\quad (\lambda,\alpha)= 0.
\end{cases}\end{equation}

From~(\ref{lambdalambda'}), it follows that $L_{\Sigma}(\lambda)$ is integrable
if and only if  $(\lambda,\alpha)\in\mathbb{Z}_{\geq 0}$ for every even $\alpha\in\Sigma$, and for two odd roots
$\beta_1,\beta_2\in\Sigma$ one has
either $(\lambda,\beta_1+\beta_2)\in \mathbb{Z}_{>0}$
or $(\lambda,\beta_1)=(\lambda,\beta_2)=0$. Since $\delta$ is the sum of simple roots,
the central charge of a highest weight module $L_{\Sigma}(\lambda)$ is $(\lambda,\sum_{\alpha\in\Sigma}\alpha)$.
In particular, if $L_{\Sigma}(\lambda)$ is not one-dimensional, its central charge is a positive integer.

\subsubsection{Example} Let $\Sigma$ be as in (\ref{sigma}) and let $b_i=(\lambda+\rho_\Sigma,\varepsilon_i)$ for $i=1,\dots, n$. Then $L_{\Sigma}(\lambda)$ is integrable
if and only if $b_1-b_2,\dots,b_{n-1}-b_n\in\mathbb Z_{>0}$ and $b_n-b_1+n+k-1\in\mathbb Z_{>0}$ or $b_n=0,b_1=n+k-1$.

\subsection{}
\begin{lem}{integrable} If  $L_{\Sigma}(\lambda)$ is integrable, then there exists at most one positive real even root $\alpha$ such that
$(\lambda+\rho_\Sigma,\alpha)=0$. Moreover, in this case $\alpha\in\Pi_0$ and
$\alpha$ is a sum of the two odd roots  $\beta_1,\beta_2\in\Sigma$ and
$(\lambda+\rho,\beta_1)=(\lambda+\rho,\beta_2)=0$.
\end{lem}
\begin{proof} If $\beta_1,\beta_2$ denote two odd roots of $\Sigma$, then integrability condition implies that $(\lambda+\rho_\Sigma,\alpha)$ is a positive integer
for all even $\alpha\in\Sigma$ and $(\lambda+\rho,\beta_1+\beta_2)$ is a non-negative integer. Since every positive even real root is a non-negative
linear combination of even $\alpha\in\Sigma$ and $\beta_1+\beta_2$ the statement follows.
\end{proof}

\subsubsection{}
\label{corLambdak}
We fix a set of simple roots $\Sigma=\{\alpha_i\}_{i=0}^n$, where $\alpha_1,\alpha_2$ are odd. Note that $(\alpha_1,\alpha_2)=1$. We set
$$\begin{array}{ll}
P^+(\Sigma):=&\{\lambda\in\fh^*|\ (\lambda,\gamma)\in\mathbb{Z}_{\geq 0}\ \text{ for all } \gamma\in\Pi_0\ \text{ and }\\
& (\lambda,\alpha_1+\alpha_2)=0\ \Longrightarrow\ (\lambda,\alpha_1)=(\lambda,\alpha_2)=0\}.\end{array}$$
By above, the irreducible objects of $\cF_k$ are the highest weight modules $L_{\Sigma}(\lambda)$, where $\lambda\in P^+(\Sigma)$; in other words,
setting $a_i:=(\rho wt_{\Sigma} L,\alpha_i)$, we have
\begin{enumerate}
\item $a_i\in \mathbb{Z}_{>0}$ for $i=0$ or $i=3,\dots,n$;
\item $a_1+a_2\in\mathbb{Z}_{>0}$ or $a_1=a_2=0$;
\item $a_0+a_1+\dots+a_n=k+n-1$.
\end{enumerate}

Notice that the numbers $\{a_i\}_{i=0}^n$ determines (up to isomorphism) $L$ as $[\fg,\fg]$-module;
for the  $\fg$-modules $L(\lambda), L(\lambda+s\delta)$ the numbers   $\{a_i\}_{i=0}^n$ are the same, however the Casimir element
acts on  $L(\lambda)$ and on  $L(\lambda+s\delta)$ by different scalars.

\subsubsection{}
\begin{lem}{positiveSigma}
Let $L_{\Sigma}(L)$ be an integrable highest weight module
and $(\lambda,\alpha)$ is real for each $\alpha\in\Sigma$.
 Then
there exists a set of simple roots $\Sigma'$ such that
$(\lambda+\rho_{\Sigma},\alpha)\geq 0$ for every $\alpha\in\Sigma'$.
\end{lem}
\begin{proof}
Recall that $(\lambda+\rho_{\Sigma},\delta)=k+n-1$.
Note that $(\Sigma\setminus\{\alpha_1,\alpha_2\})\cup\{ \alpha_1+s\delta,\alpha_2-s\delta\}$
is a set of simple roots for any $s\in\mathbb{Z}$. Therefore without loss of generality we may assume that
\begin{equation}\label{beta1k}
0\leq (\lambda+\rho_\Sigma,\alpha_1)<k+n-1.
\end{equation}
If $0\leq (\lambda+\rho_\Sigma,\alpha_2)$, we take
$\Sigma'=\Sigma$.
Assume that $(\lambda+\rho_\Sigma,\alpha_2)<0$. For $r=2,\ldots,n+1$ set
 $\beta_r:=\sum_{i=2}^{r}\alpha_i$ (where $\alpha_{n+1}:=\alpha_0$).
Then $\delta=\beta_{n+1}+\alpha_1$, so ~(\ref{beta1k}) gives
$$(\lambda+\rho_\Sigma,\beta_2)<0,\ \ (\lambda+\rho_\Sigma,\beta_{n+1})>0.$$
Let $s$ be maximal such that $(\lambda+\rho_\Sigma,\beta_s)<0$.
For $\Sigma':=r_{\beta_s}\ldots r_{\beta_2} \Sigma$ the isotropic roots are
$-\beta_s$ and $\beta_{s+1}$. Since  $(\lambda+\rho_\Sigma,-\beta_s), (\lambda+\rho_\Sigma,\beta_{s+1})\geq 0$,
$\Sigma'$ is as required.
\end{proof}

\subsubsection{Definitions}
From now on $L$ stands for an {\em irreducible integrable highest weight module
on non-zero level}.

Recall that $L$ is called {\em typical}
if $(\rho wt_{\Sigma} L,\alpha)\not=0$ for any (isotropic) odd root
$\alpha$ and {\em atypical} otherwise. From~(\ref{lambdalambda'}), it follows that this notion does not depend on the choice
of $\Sigma$ and, moreover,
  $\rho wt_{\Sigma} L$ does not depend on $\Sigma$ for typical $L$.

We say that $L$ is {\em $\Sigma$-tame}
if $(\rho wt_{\Sigma} L,\beta)=0$ for some odd $\beta\in\Sigma$.
Any atypical $L$ (for $\fsl(1,n)^{(1)}$) is tame with respect to some $\Sigma$.

Let $\beta$ be an odd root. We call an odd reflection $r_{\beta}$  {\em $L$-typical} if
for $\Sigma$ containing $\beta$ one has
$(\rho wt_{\Sigma} L,\beta)\not=0$ (by~\ref{uniqueSigma}, $\Sigma$ is unique for given
$\beta$).
Note that if $\Sigma$ and $\Sigma'$ are connected by a chain of odd
$L$-typical reflections, then
$\rho wt_{\Sigma}(L)=\rho wt_{\Sigma'}(L)$.

We say that  $\lambda\in\fh^*$ is {\em integral} if $(\lambda,\alpha)$ is integral
for each $\alpha\in\Pi_0$.
We say that $\lambda\in\fh^*$ is {\em regular} if $(\lambda,\alpha)\not=0$
for any even real root and that $\lambda$ is {\em singular} otherwise.
Note that if $\lambda$ integral, then there exists a unique $\bar{\lambda}\in W\lambda$
such that $(\bar{\lambda},\alpha)\in\mathbb N$ for all $\alpha\in \Pi_0$. Obviously $\bar{\lambda}$ is regular if and only if $\lambda$ is regular.

We say that $L$ is {\em $\Sigma$-regular} if $\rho wt_{\Sigma} L$ is regular  and
that $L$ is {\em regular} if it is $\Sigma$-regular for each $\Sigma$. We say that $L$ is {\em $\Sigma$-singular} if it is
not $\Sigma$-regular and that $L$ is {\em singular} if it is not regular.
By~\ref{corLambdak}, $L$ is $\Sigma$-singular if and only if $(\rho wt L,\alpha)=0$
for both odd roots $\alpha\in\Sigma$ (in particular, in this case $L$ is $\Sigma$-tame).

\subsubsection{}
\begin{lem}{hw} Let $L$ be a simple integrable module and $HW(L)$ denote the set of all $\rho wt_{\Sigma} L$ for all $\Sigma\in\mathcal B$.
\begin{enumerate}
\item If $L$ is typical then $|HW(L)|=1$;
\item If $L$ is regular atypical then $|HW(L)|=2$;
\item Let $L$ be singular atypical and $\Sigma=\{\alpha_0,\dots,\alpha_{n}\}$ with odd $\alpha_0,\alpha_1$ be such that $\rho wt_{\Sigma} L$ is singular.
Let $m,l$ be such that
$$(\rho wt_\Sigma  L,\alpha_m)\geq 2,\,(\rho wt_\Sigma L,\alpha_{m+1})=\dots=(\rho wt_\Sigma L,\alpha_{n})=1, \quad $$
and
$$(\rho wt_\Sigma L,\alpha_l)\geq 2,\,(\rho wt_\Sigma L,\alpha_{2})=\dots=(\rho wt_\Sigma L,\alpha_{l-1})=1.$$
Then
$$\begin{array}{ll}
HW(L)=&\{\rho wt_{\Sigma}L, \rho wt_{\Sigma}L+\alpha_1,\rho wt_{\Sigma}L+2\alpha_1+
\alpha_2,\dots,\\
& \rho wt_{\Sigma}L+(l-1)\alpha_1+(l-2)\alpha_2+\cdots+\alpha_{l-1}\}
\cup\{\rho wt_{\Sigma}L+\alpha_0,\dots,\\
&\rho wt_{\Sigma}L+(n-m+1)\alpha_0+(n-m)\alpha_n+\dots+\alpha_{m+1}
\}\end{array}$$
or equivalently
$$HW(L)=\{\overline{\rho wt_{\Sigma}L+s\alpha_0}\,|\,s=1-l,\dots,n-m+1\}.$$
\end{enumerate}
\end{lem}
\begin{rem}{}
The existence of $l,m$ follows from the assumption that $k\not=0$; one has
 $2\leq l\leq m\leq n$.
\end{rem}
\begin{proof} The first assertion follows from (\ref{lambdalambda'}) as any odd reflection is $L$ typical.
Now assume that $L$ is atypical and regular. Then there exists exactly one positive odd root such that $(\rho wt_{\Sigma} L,\alpha)=0$.
Let $\Sigma'\in\mathcal B$ be such that $\alpha\in\Sigma'$. Let $\Sigma''=r_{\alpha}\Sigma'$. Since both $\rho wt_{\Sigma'} L$ and $\rho wt_{\Sigma''} L$ are regular,
all other odd reflections are $L$-typical. Hence $HW(L)=\{\rho wt_{\Sigma'} L,\rho wt_{\Sigma''} L\}$.

Finally let assume that $L$ is atypical and singular.
Then by Lemma~\ref{integrable} there exists  $\Sigma_0$ such that
$(\rho wt_{\Sigma} L,\alpha_1)=(\rho wt_{\Sigma} L,\alpha_2)=0$ for both simple odd roots $\alpha_1,\alpha_2\in\Sigma_0$.
Moreover, $m$ and $l$ exist as follows from integrability condition.
If we set
$$\beta_1=\alpha_0+\alpha_n,\dots,\beta_{n-m}=\alpha_0+\alpha_n+\dots+\alpha_{n-m+1},$$
then the reflection $r_{\beta_i}$ are all $L$-atypical and we obtain that
$HW(L)$ contains $\rho wt_{r_{\alpha_0}\Sigma} L$ and
$\rho wt_{r_{\beta_i}\dots r_{\beta_1}r_{\alpha_0}\Sigma}L$ for all $i=1,\dots, n-m$.
Similarly if we set
$$\gamma_1=\alpha_1+\alpha_2,\dots,\gamma_{l-2}=\alpha_1+\alpha_2+\dots+\alpha_{l-1},$$
then $HW(L)$ contains $\rho wt_{r_{\alpha_1}\Sigma} L$ and $\rho wt_{r_{\gamma_i}\dots r_{\gamma_1}r_{\alpha_1}\Sigma}L$ for all $i=1,\dots, l-2$.
All other odd reflections are $L$ typical and do not add new weights to $HW(L)$.

The last formula follows form the identity
$$r_{\alpha_j}\ldots r_{\alpha_2}r_{\alpha_1+\alpha_0}(\rho wt_{\Sigma}L-j\alpha_0)=
r_{\alpha_j}\ldots r_{\alpha_2}(\rho wt_{\Sigma}L+j\alpha_1)=
\rho wt_{\Sigma}L+j\alpha_1+(j-1)\alpha_2+\dots+\alpha_j.$$
\end{proof}

\subsubsection{}
One readily sees that $\rho wt_{\Sigma}L+j\alpha_0$ is not regular for
$1-l<j<n-m+1$.

\begin{cor}{corhw} If $L$ is atypical, then $HW(L)$ contains exactly two regular weights.
\end{cor}
\subsection{Character formulae}\label{chformula}
If $L(\lambda)$
is typical, then $\ch L(\lambda)$ is given by the Kac-Weyl character formula;
if $L(\lambda)$ is atypical and $\Sigma$-tame,  $\ch L(\lambda)$ is given by Kac-Wakimoto formula, see~\cite{S2},\cite{KW2}.

\subsection{Adjacency relation on atypical simple modules}
Note that if $L$ is atypical, then all weights of $L$ are integral. We say that $L'$ is adjacent to $L$ if there exist $\Sigma\in\mathcal B$ and an odd
$\alpha\in\Sigma$ such that $(\rho wt_\Sigma L,\alpha)=0$ and $\rho wt_\Sigma L'=\rho wt_\Sigma L-\alpha$. Note the if $L'$ is adjacent to $L$, then $L$ is adjacent to
$L'$, as for $\Sigma'=r_{\alpha}\Sigma$ we have  $\rho wt_{\Sigma'} L=\rho wt_{\Sigma'} L'-(-\alpha)$ and $-\alpha\in\Sigma'$.
Therefore the adjacency relation defines the adjacency graph with vertices enumerated by isomorphism classes of atypical integrable simple $\fg$-modules with a fixed central charge $k$ and a fixed eigenvalue of the Casimir operator
(i.e., the value $(\rho wt_{\Sigma}L,\rho wt_{\Sigma}L)$ is fixed).

We denote this graph by $\Gamma$. It is important to characterize the connected components of $\Gamma$.

\subsubsection{}
\begin{lem}{onlyreg} Let $\Sigma\in\mathcal B$ and $\alpha$ be an odd root of $\Sigma$ such that $(\rho wt_\Sigma L,\alpha)=0$. Then
$\rho wt_\Sigma L-\alpha$ is integrable if and only if $\rho wt_\Sigma L$ is regular.
\end{lem}
\begin{proof} If $\beta\in\Sigma$ is an even root, then $(-\alpha,\beta)\geq 0$ and hence $(\lambda-\alpha,\beta)\in\mathbb Z_{>0}$.
If $\beta\in\Sigma$ is the second odd root, then $(\alpha+\beta,\alpha)=1$ and therefore $(\lambda-\alpha,\beta+\alpha)\in\mathbb Z_{\geq 0}$ if and only
if $(\lambda,\beta)\geq 1$. Hence $\lambda$ is regular.
\end{proof}

\subsubsection{}
\begin{cor}{onlytwo} An atypical simple integrable module $L$ has exactly two adjacent $L'$ and $L''$.
To construct them recall that by Lemma~\ref{hw} there exist exactly two $\Sigma_1$ and $\Sigma_2$ in $\mathcal B$ such that $L$ is $\Sigma_i$-tame and
$\rho wt_{\Sigma_i}L$ is regular.
Then $\rho wt_{\Sigma_1}L'=\rho wt_{\Sigma_1}-\alpha_1$ and $\rho wt_{\Sigma_2}L''=\rho wt_{\Sigma_2}-\alpha_2$ where $\alpha_i$ is the unique odd root in $\Sigma_i$
such that
$(\rho wt_{\Sigma_i} L,\alpha_i)=0$.
\end{cor}

\subsubsection{}
\begin{rem}{order} Let us fix $\Sigma\in\mathcal B$, it follows from the above corollary that
$$\rho wt_\Sigma L'<_\Sigma\rho wt_\Sigma L<_\Sigma\rho wt_\Sigma L''$$ if $\rho wt_\Sigma L$ is regular. If  $\rho wt_\Sigma L$ is singular, we have
$$\rho wt_\Sigma L<_\Sigma\rho wt_\Sigma L'\,\,\text{and}\,\,\rho wt_\Sigma L<_\Sigma\rho wt_\Sigma L''.$$
\end{rem}

\subsubsection{}
\begin{thm}{maingraph} Fix $\Sigma\in\mathcal B$. Then every connected component $\Gamma'$ of $\Gamma$ contains exactly one $L_0$ such that
$\lambda:=\rho wt_\Sigma L_0$ is singular. Let $\beta$ be any of two odd roots of $\Sigma$ and
$$S:=\{s\in\mathbb Z\,|\,\lambda+s\beta\,\,\text{is regular}\}.$$
Then $L'\in\Gamma'$ if and if $L_0\simeq L'$ or $\rho wt_\Sigma L'=\overline{\lambda+s\beta}$ for some $s\in S$. Enumerate elements of $S\cup\{0\}$ in increasing
order assuming $s_0=0$ and set $L_i:=L_{\Sigma}(\overline{\lambda+s_i\beta}-\rho_\Sigma)$. Then every $L_i$ is adjacent to $L_{i-1}$ and $L_{i+1}$.
\end{thm}
\begin{proof} Uniquness of $L_0$ follows from Remark \ref{order} and Corollary \ref{onlytwo}. Let us prove the existence. Start with some $L$ such
that $\rho wt_\Sigma L$ is regular. There exists
$\Sigma'$ such that $L$ is $\Sigma'$-tame. Let us pick up $L$ with minimal $d(\Sigma,\Sigma')$. We claim that for such $L$, $\Sigma=\Sigma'$. Indeed, assume
$\Sigma\neq\Sigma'$. By Lemma \ref{hw}, $\mu=\rho wt_{\Sigma}L=\rho wt_{\Sigma'}L$ is regular. There is a unique odd $\alpha\in\Sigma'$ for which
$(\mu,\alpha)=0$. Consider the smallest $p>0$ such that $\mu-p\alpha$ is not dominant. Then $\mu-(p-1)\alpha$ is singular. Let
$L'=L_{\Sigma'}(\mu-(p-1)\alpha-\rho_{\Sigma'})$. If $\alpha'$ is the second odd root of $\Sigma'$, then $d(\Sigma,r_{\alpha'}\Sigma)=d(\Sigma,\Sigma')-1$ and
this contradicts minimality of
$d(\Sigma,\Sigma')$. To finish the proof of existence of $L_0$ take odd $\beta\in\Sigma$ such that $(\mu,\beta)=0$ and consider
the smallest $q\geq 0$ such that $\mu-q\beta$ is singular.
Then $L_0$ is the simple module with $\rho wt_{\Sigma}L_0=\mu-q\beta$.

The last assertion of the theorem follows fom the description of $HW(L)$ given in Lemma \ref{hw}.
\end{proof}

\subsubsection{}
\begin{cor}{paramcon} For a fixed central charge $k$, the graph $\Gamma$ has finitely many connected components. They are enumerated by singular integrable weights
with central charge $k$.
\end{cor}

\section{The category of integrable $sl(1|n)^{(1)}$-modules
 with positive central charge}
 In this section we will describe $\cF_k$ for $k>0$.

\subsection{Maximal integrable quotient of a Verma module}\label{maxint}
Let $\Sigma\in\mathcal B$. We denote by $M_{\Sigma}(\lambda)$
the Verma module with  highest weight $\lambda$ for the Borel subalgebra corresponding to
$\Sigma$. The Verma module $M_{\Sigma}(\lambda)$ has a unique simple quotient $L_\Sigma(\lambda)$. If  $L_\Sigma(\lambda)$ is integrable, then
we denote by  $V_\Sigma(\lambda)$ the maximal integrable quotient of $M_{\Sigma}(\lambda)$. Clearly we have a surjection $V_{\Sigma}(\lambda)\to L_\Sigma(\lambda)$.

In~\cite{S2} the following lemma is proved.

\subsubsection{}
\begin{lem}{integrablequotient} Let $L=L_\Sigma(\lambda-\rho_\Sigma)$ be an  integrable module.

(i) If $L$ is typical, then $V_\Sigma(\lambda-\rho_\Sigma)=L$.

(ii) If $L$ is atypical and $\lambda$ is singular, then  $V_\Sigma(\lambda-\rho_\Sigma)=L$.

(iii) If $\lambda$ is regular, then the character of $V_{\Sigma}(\lambda-\rho_\Sigma)$ is given by typical formula
$$\ch V_\Sigma(\lambda-\rho_\Sigma)=\sum_{w\in W}\operatorname{sgn}(w)\ch M_\Sigma(w(\lambda)-\rho_\Sigma).$$
Moreover if $(\lambda,\alpha)=0$ for some odd $\alpha\in\Sigma$, then $V_\Sigma(\lambda-\rho_\Sigma)$ has length two
and can be described by the following exact sequence
$$0\to L_\Sigma(\lambda-\alpha-\rho_\Sigma)\to V_\Sigma(\lambda-\rho_\Sigma)\to L_\Sigma(\lambda-\rho_\Sigma)\to 0.$$

(iv) For any $\Sigma$ and $\Sigma'$ in $\mathcal B$, such that $\lambda=\rho wt_\Sigma L=\rho wt_{\Sigma'} L$, we have
$V_\Sigma(\lambda-\rho_\Sigma)=V_{\Sigma'}(\lambda-\rho_{\Sigma'})$
\end{lem}

\subsubsection{}
\begin{lem}{corext} Let $L$ and $L'$ be two non-isomoprhic simple integrable modules. Then $\Ext^1(L,L')\neq 0$ if and only of $L$ and $L'$ are two adjacent
atypical modules. In this case $\Ext^1(L,L')=\mathbb C$.
\end{lem}
\begin{proof} Consider an extension given by a non-split exact sequence
$$0\to L'\to M\to L\to 0.$$
Choose some  $\Sigma\in\mathcal B$ and let $\lambda= \rho wt_{\Sigma} L$, $\mu=\rho wt_{\Sigma} L'$. If $\lambda$ and $\mu$ are incomparable with respect to
$\leq_{\Sigma}$, then the above sequence splits since a vector of weight $\lambda-\rho_\Sigma$ generates a submodule isomorphic to $L$ in $M$.
Note that duality implies $\Ext^1(L,L')=\Ext^1(L',L)$. Therefore without loss of generality we may assume that $\mu<_{\Sigma}\lambda$. But then $M$ is a quotient
of $V_{\Sigma}(\lambda-\rho_\Sigma)$. By Lemma \ref{integrablequotient} we know that the length of $V_{\Sigma}(\lambda-\rho_\Sigma)$ is at most $2$. Hence
$M\simeq V_{\Sigma}(\lambda-\rho_\Sigma)$, $L$ is atypical and $\lambda$ is regular. Then there exists $\Sigma'\in\mathcal B$ such that $L$ is $\Sigma'$-tame
and $\rho wt_\Sigma L=\rho wt_{\Sigma'} L$. By Lemma \ref{integrablequotient} (iv) we obtain $\rho wt_{\Sigma'} L'=\rho wt_{\Sigma'} L-\alpha$ for odd $\alpha\in\Sigma'$
such that $(\lambda,\alpha)$. Hence by definition $L$ and $L'$ are adjacent.
That proves the statement.
\end{proof}

\subsection{Self-extensions of simple modules}
\subsubsection{}\label{top}
Recall that $\fg$ is the affinization of $\dot{\fg}=\fsl(1|n)$. Fix
$\Sigma\in\mathcal{B}$ and $\alpha\in\Sigma$.
Note that $\Sigma\setminus\{\alpha\}$ is the set of simple roots of some
subalgebra isomorphic to $\dot\fg\simeq\fsl(1|n)$. Let $h\in\fh^*$ be such that
$$\beta(h)=\begin{cases} 0\,\,\text{if}\,\,\beta\in\Sigma,\beta\neq\alpha\\1 \,\,\text{if}\,\,\beta=\alpha\end{cases}.$$

Let $N$ be such that $h$ acts locally finitely and the eigenvalues are bounded:
there exist a ``maximal'' eigenvalue, i.e.
an eigenvalue $a$ such that $a+j$ is not an eigenvalue
for any positive integer $j$. In this case we denote by $N^{top}$
be the generalized $h$-eigenspace with the maximal eigenvalue.

Observe that if $L=L_{\Sigma}(\lambda)$ is simple, then $L^{top}$
is a simple $\fsl(1|n)$-module with highest weight $\lambda|_{\dot{\fh}}$ where ${\dot{\fh}}$ is the Cartan subalgebra of ${\dot{\fg}}$ .
If $N$ is integrable, then $N^{top}$ is finite-dimensional.

The centre  of $\fsl(1|n)^{(1)}_{\ol{0}}$ is
two-dimensional: it is spanned by $K$ and $z$, where $z$ is a central element
in $\mathfrak{sl}(1|n)_{\ol{0}}=\fgl_n$.

\subsubsection{}
\begin{lem}{applself1} Let $L$ be a simple module and
$$0\to L\to M\to L\to 0$$
be a non-split exact sequence, then the sequence
$$0\to L^{top}\to M^{top}\to L^{top}\to 0$$
also does not split.
\end{lem}
\begin{proof} If $M^{top}\simeq L^{top}\oplus L^{top}$, the two copies of $M^{top}$ generate two proper distinct submodules of $M$.
Since $M$ has length $2$, it is a direct sum of two simple modules.
\end{proof}

\subsection{} Recall now the following result from \cite{G}.

\begin{lem}{Germoni} If $N$ is a simple $\dot{\fg}$-module, then $\Ext^1(N,N)=0$ if $N$ is atypical and $\Ext^1(N,N)=\mathbb C$ if $N$ is typical.
\end{lem}

\subsubsection{}
\begin{cor}{atypicalself} For a simple atypical $\fg$-module $L$, $\Ext^1(L,L)=0$.
\end{cor}
\begin{proof} It is easy to find $\Sigma$ and $\alpha\in\Sigma$ such that $L^{top}$ is atypical. Then the statement follows from Lemma \ref{applself1} and
Lemma \ref{Germoni}.
\end{proof}

\subsubsection{} In \cite{S2} the following statement is proved (Lemma 4.13.)
\begin{lem}{deformation} Let $L$ be a simple module and $\Sigma\in\mathcal B$ be such that $\rho wt_\Sigma L$ is regular. Let $\omega$ be a weight such that
$(\omega,\alpha)=1$ for some odd $\alpha\in\Sigma$ and $(\omega,\beta)=-1$ for another odd $\beta\in\Sigma$ and
$(\omega,\gamma)=0$ for all even $\gamma\in\Sigma$.
Then $V_\Sigma(\lambda-\rho_\Sigma+t\omega)$ is a flat deformation of $V_\Sigma(\lambda-\rho_\Sigma)$.
\end{lem}
\subsubsection{}
\begin{cor}{intself} Under assumptions of the previous lemma the module $V_\Sigma(\lambda-\rho_\Sigma+t\omega)/(t^p)$ is an indecomposable module which has a
filtration with associated quotients ismorphic $V_\Sigma(\lambda-\rho_\Sigma)$.
\end{cor}

\subsection{Typical blocks in $\cF_k$}

Let $\dot{L}$ be a typical finite-dimensional $\mathfrak{sl}(1|n)$-module of highest weight $\dot{\lambda}$
and let  $\cF(\dot{L})$ be the  block containing $\dot{L}$
in the category  of finitely generated $\mathfrak{sl}(1|n)$-modules.
It is easy to deduce from~\cite{G} that  the functor $N\mapsto N_{\lambda}$ (here $N_{\lambda}$ is the subspace with generalized weight $\lambda$)
provides an equivalence between $\cF(\dot{L})$
and the category of finite-dimensional  $\mathbb{C}[z]$-modules with nilpotent action of $z-\lambda(z)$.

\subsubsection{} Retain notation of~\S~\ref{top}.

\begin{cor}{cortyp}
For any typical simple module $L$ in $\cF_k$ there exists a block $\cF_k(L)$ of $\cF_k$
which has one up to isomorphism simple module $L$.
The functor $N\mapsto N^{top}$ provides an equivalence
between $\cF_k(L)$ and the typical block of the category of finitely generated $\mathfrak{sl}(1|n)$-modules.
The functor $N\to N_{\lambda}$ provides an equivalence
between $\cF_k(L)$ and the category of finite-dimensional  $\mathbb{C}[z]$-modules with nilpotent action
of $z-\lambda(z)$.
\end{cor}

\subsection{Atypical blocks of $\cF_k$}
\subsubsection{} The following theorem is a direct consequence of Lemma \ref{corext}, Corollary \ref{atypicalself} and Theorem \ref{maingraph}.

\begin{thm}{thm-ext} The Ext quiver of an atypical block in $\cF_k$ coincides with  a connected component of the graph $\Gamma$
and is of the form
$$
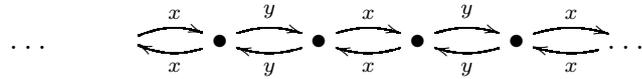
\xymatrix{\dots& \ar@/^0.4pc/[r]^{x}&
\bullet\ar@/^0.4pc/[r]^{y}\ar@/^0.4pc/[l]^{x} & \bullet \ar@/^0.4pc/[r]^{x} \ar@/^0.4pc/[l]^{y}&
\bullet\ar@/^0.4pc/[l]^{x}\ar@/^0.4pc/[r]^{y} & \bullet
\ar@/^0.4pc/[r]^{x}\ar@/^0.4pc/[l]^{y}&
\ar@/^0.4pc/[l]^{x}\dots}$$
\end{thm}

\subsubsection{}
Let us show that the above quiver satisfies the relations $xy=yx=0$.

\begin{lem}{length3} There is no indecomposable module $M$ in $\cF_k$ such that
$M/rad M=L_1$, $rad M/rad^2M=L_2$, $rad^2M=L_3$
for pairwise non-isomorphic simple modules $L_1,L_2,L_3$.
\end{lem}
\begin{proof} Assume that such module exists. Then $\{L_1,L_2,L_3\}$ generate a connected subgraph of $\Gamma$.
It follows from Theorem~\ref{maingraph} and Lemma~\ref{hw} that there exists $\Sigma$ and $\alpha\in\Sigma$ such that $L_i^{\top}$ has the same $h$-eigenvalue for
$i=1,2,3$ (see~\S~\ref{top} for the notation $h$). It was shown in \cite{G} that there is no similar indecomposable $\dot{\fg}$-module $M^{top}$ with
$M^{top}/rad M^{top}=L_1^{top}$, $rad M^{top}/rad^2M^{top}=L_2$, $rad^2M^{top}=L_3^{top}$. The statement follows.
\end{proof}

\subsubsection{}
\begin{thm}{thmatyp} Any atypical block in $\cF_k$ is equivalent to the category of finite-dimensional  representations of the quiver of Theorem~\ref{thm-ext}
with relations $xy=yx=0$.
\end{thm}
\begin{proof} We use the same argument as in the proof of the previous lemma. The relation $xy=yx=0$ implies that the simple subquotients of any
indecomposable module
are isomorphic to $L_i,L_{i+1},L_{i+2}$ for some $i$. Consider the full subcategory $\mathcal C$ of $\cF_k$ which contains only modules with
semisimple subquotients isomorphic to  $L_i,L_{i+1}$ or $L_{i+2}$ and let ${\mathcal C}'$ be the full subcategory of $\dot{\fg}$-modules  which contains only
modules with
semisimple subquotients isomorphic to  $L^{top}_i,L^{top}_{i+1}$ or $L^{top}_{i+2}$. We claim that the functor $?^{top}$ defines an equivalence between $\mathcal C$ and
${\mathcal C}'$. Indeed, $?^{top}$ is exact and provides a bijection on isomorphism classes of simple modules.
To construct the left adjoint functor $\Phi$ consider the parabolic subalgebra $\fp:=\fb+\dot{\fg}$ and set
$\Phi(?)$ to be the maximal quotient of $U(\fg)\otimes_{U(\fp)}?$ which lie in $\mathcal C$. We leave it to the reader to check that $\Phi$ is also exact.
Now theorem follows from the analogous results in~\cite{G} for $\dot{\fg}$.
\end{proof}

\subsection{}
Let $\cF^1_k$ be the full subcategory of $\cF$ consisting of the modules with the diagonal action of $\fh$. Using the results of~\cite{G} and the proof of~\Thm{thmatyp}
we obtain the following result.

\begin{cor}{corF1}
The typical blocks in $\cF_1^k$ are completely reducible with a unique irreducible module. Any atypical block in $\cF_k$ is equivalent to the category of finite-dimensional  representations of the quiver of Theorem~\ref{thm-ext}
with relations $xy=yx=0$ and $x^2+y^2=0$.
\end{cor}

\subsubsection{}
Retain notation of~\S~\ref{top}.
If $N$ is a $\fg$-module and $N^{top}$ is well-defined, $N$ is called {\em almost irreducible} if any nontrivial
submodule of $N$ has a non-zero intersection with $N^{top}$.

Note that $\cF^k$ is a subcategory in the category $\tilde{\CO}$
(this category is defined as the category $\CO$, but the action of $\fh$ is assumed
to be locally finite). In particular, $N^{top}$ is well-defined for any choice
of $\Sigma$ and $\alpha$.

Let $N$ be an indecomposable module. If $N$ lies in a typical block,
then $N$ is almost irreducible  for any choice of $\Sigma$ and $\alpha$.
If $N$ lies in an atypical block, then $N$ is almost irreducible for
$\Sigma,\alpha$ as in the proof of~\Thm{thmatyp}.

\section{The functor $F_x$}\label{Fx}
In this section we assume that $\fg$ is a Kac-Moody Lie  superalgebra.

Take $x\in\fg_{\ol{1}}$ satisfying $[x,x]=0$.
The following construction is due to  M.~Duflo and V.~Serganova, see~\cite{DS}.
For a $\fg$-module $N$ introduce
$$F_x(N):=Ker_N x/Im_N x.$$

Let $\fg^x$ be the centralizer of $x$ in $\fg$.
We view $F_x(N)$ as a module over $\fg^x$.
Note that $[x,\fg]\subset \fg^x$ acts trivially on $F_x(N)$ and
that $\fg_x:=F_x(\fg)=\fg^x/[x,\fg]$ is a Lie superalgebra. Thus
$F_x(N)$ is a $\fg_x$-module and $F_x$ is a functor from the category of $\fg$-modules to the category of $\fg_x$-modules.

In~\cite{DS},\cite{S1} the functor $F_x$ was studied for finite-dimensional $\fg$. However, certain properties can be easily generalized to the affine case.
In particular, $F_x$ is a tensor functor, i.e. there is a canonical isomorphism $F_x(N_1\otimes N_2)\simeq F_x(N_1)\otimes F_x(N_2)$.

\subsection{}
\begin{prop}{gx} Let $\fg=\dot{\fg}^{(1)}$ be the affinization of a Lie superalgebra $\dot{\fg}$ and assume that $x\in\dot{\fg}$. If $\dot{\fg_x}\neq 0$, then
$\fg_x$ is the affinization of $\dot{\fg}_x$, If $\dot{\fg}_x=0$ then $\fg_x$ is the abelian two-dimensional Lie algebra generated by $K$ and $d$.
\end{prop}
\begin{proof} Since $$\fg=\mathbb Cd\oplus\mathbb CK\oplus\bigoplus_{n\in \mathbb Z}\dot{\fg}\otimes t^n$$
and $\dot{\fg}\otimes t^n$ is isomorphic to the adjoint representation of $\dot{\fg}$ for every $n$, the statement follows.
\end{proof}

\subsection{}
Let $\fg=\dot{\fg}^{(1)}$ be the affinization of a Lie superalgebra $\dot{\fg}$ and assume that $x\in\dot{\fg}$.
Let $\dot{\Sigma}$ (resp., $\Sigma$) be the set of simple roots of $\dot{\fg}$
(resp., $\fg$).

Let $\beta_1,\dots\beta_r\in\dot{\Sigma}$ be a set of mutually orthogonal isotopic simple roots, fix  non-zero root vectors
$x_i\in\fg_{\beta_i}$ for all $i=1,\dots,r$. Let $x:=x_1+\dots+x_r$. It is shown in ~\cite{DS} that $\dot{\fg}_x$ is a finite-dimensional Kac-Moody superalgebra with
roots
$$\dot{\Delta}^{\perp}:=\{\alpha\in\dot{\Delta}|\ (\alpha,\beta_i)=0, \alpha\not=\pm\beta_i\, i=1,\dots,r\}$$
and the Cartan subalgebra
$$\fh_x:=(\beta_1^\perp\cap\dots\cap\beta_r^\perp)/(\mathbb C h_{\beta_1}\oplus\dots\oplus \mathbb C h_{\beta_r}).$$
Assume that $\dot{\Delta}^{\perp}$ is not empty, then
$\dot{\Delta}^{\perp}$ is the root system of the Lie superalgebra $\dot{\fg}_x$.
One can choose a set of simple roots $\dot{\Sigma}_x$ such that $\Delta^+(\dot{\Sigma}_x)=\Delta^+\cap \dot{\Delta}^{\perp}$.
Let $\fg_x\subset\fg$ be the affinization of  $\dot{\fg}_x$:
the affine Lie superalgebra
with a set of simple roots $\Sigma_x$ containing $\dot{\Sigma}_x$
such that $\Delta^+({\Sigma}_x)\subset\Delta^+$.

For example, if $\dot{\fg}=A(m|n), B(m|n)$ or $D(m|n)$, then $\dot{\fg}=A(m-r|n-r), B(m-r|n-r)$ or $D(m-r|n-r)$.
If $\dot{\fg}=C(n)$, $G_3$ or $F_4$, then $r=1$ and $\dot{\fg_x}$ is the Lie algebra of type $C_{n-1}$, $A_1$ and $A_2$ respectively.
If  $\dot{\fg}=D(2,1;\alpha)$, then $r=1$ and $\fg_x=\mathbb C$.

\subsection{}
\begin{prop}{lemCasimir}
Let $\fg=\dot{\fg}^{(1)}$ be the affinization of a Lie superalgebra $\dot{\fg}$ and assume that $x\in\dot{\fg}$.
Let $x\in\dot{\fg}$ and $N$ be a restricted $\fg$-module.
If the Casimir  element $\Omega_{\fg}$ acts on a $N$
by a scalar $C$, then the Casimir element $\Omega_{\fg_x}$  acts on the $\fg_x$-module
$F_x(N)$ by the same scalar $C$.
\end{prop}
\begin{proof} Let us write the Casimir element $\Omega_{\fg}$ in the following form (see~\cite{K3}, (12.8.3))
$$\Omega_{\fg}=2(h^{\vee}+K)d+\Omega_0+2\sum_{i=1}^\infty \Omega(i),$$
where $\Omega(i)=\sum v_j v^j$ for some basis $\{v_j\}$ in $\dot{\fg}\otimes t^{-i}$ and the dual basis  $\{v^j\}$ in $\dot{\fg}\otimes t^{i}$.
Similarly we have
$$\Omega_{\fg_x}=2(h^{\vee}+K)d+\Omega_0+2\sum_{i=1}^\infty \Omega_x(i).$$
We claim that $\Omega_x(i)\equiv \Omega(i)(\operatorname{mod}[x,U(\fg)])$.
Indeed, we use the decomposition $\dot{\fg}=\dot{\fg}_x\oplus \fm$, where $\fm$ is a free $\mathbb C[x]$-module.
Using a suitable choice of bases we can write
$$\Omega(i)=\Omega_x(i)+\sum u_s u^s$$
for the pair of dual bases $\{u_s\}$ in $\fm\otimes t^{-i}$ and   $\{u^s\}$ in $\fm\otimes t^{i}$.
If $i>0$, then $\sum u_s u^s$ is $x$-invariant element via the embedding $\fm\otimes\fm\hookrightarrow U(\fg)$.
If $i=0$, then $\sum u_s u^s$ is $x$-invariant element via the embedding $S^2(\fm)\hookrightarrow U(\fg)$.
Since $\fm\otimes\fm$ and $S^2(\fm)$ are free $\mathbb C[x]$-modules, we obtain in both cases that $\sum u_s u^s$ lies in the image of $\operatorname{ad} x$.

Now the statement follows immediately from the fact that $[x,U(\fg)]$ annihilates $F_x(N)$.
\end{proof}

\section{Invariants of simple objects in the same  block}
Now let $\fg=\fsl(1|n)^{(1)}$ with $n>2$. Take a non-zero
$x\in \fg_{\beta}$, where
$\beta$ is an odd isotropic root; then $[x,x]=0$.

In this section we will show that for an irreducible modules $L,L'\in\cF_k$ one has

(i)  $F_x(L)=0$ if and only if $L$ is typical;

(ii) if $L$ is atypical, then  $F_x(L)\cong F_x(L')$ if and only if $L$ and $L'$ lie in the same block.

\subsection{}\label{Sigmax}
Fix a set of simple roots $\Sigma$; let $\alpha_1,\alpha_2\in\Sigma$  be odd roots.
Since for any odd root $\beta$ the orbit $W\beta$ contains either $\alpha_2$ or $-\alpha_2$, we may assume that $x\in\fg_{\alpha_2}$ or $x\in\fg_{-\alpha_2}$.
Then $\fg_x\cong \mathfrak{sl}_{n-1}^{(1)}$ with the set of simple roots
$$\Sigma_x:=\{\alpha_0,\alpha_1+\alpha_2+\alpha_3,\alpha_4,\ldots,\alpha_n\}.$$

Recall that, by~\Lem{integrable},  a Verma module $M(\lambda)$ has at most two integrable quotients: $L(\lambda)$ and $N$ such that $N/L(\lambda-\beta)=L(\lambda)$.

\subsection{}
\begin{prop}{propaga}
Let $L$ be an irreducible typical integrable highest weight module.
Then $F_x(L)=0$ for any non-zero $x\in\fg_{\beta}$,
where $\beta$ is an odd isotropic root.
\end{prop}
\begin{proof}
Set $\lambda:=\rho wt_{\Sigma} L$; since $L$ is typical, $\lambda$ does not depend
on $\Sigma$. Let $\Sigma=\{\alpha_i\}_{i=0}^n$ and $\alpha_1,\alpha_2$ are odd.
Take $x\in\fg_{\pm\alpha_2}$.
Assume that $F_x(L)\not=0$.

 Let $v\in L$ be a preimage of a highest weight vector in $F_x(L)$;
we can choose $v$ to be a weight vector of weight $\nu$. Then $(\nu,\alpha_2)=0$.
Therefore $(\lambda,\alpha)\not\in\mathbb{Z}$ for each $\alpha\in \Sigma$.

By~\Lem{positiveSigma}, we can (and will)
assume that $(\lambda,\alpha)>0$ for each $\alpha\in\Sigma$.
Set $\rho:=\rho_{\Sigma}$ and $a_i:=(\nu+\rho,\alpha_i)$ for $i=0,\ldots,n$.

Since $F_x(L)$ is $\fg_x$-integrable, and
$$\Pi_x=\{\alpha_0,\alpha_1+\alpha_2+\alpha_3,\alpha_4,\ldots,\alpha_n\},$$
one has
\begin{equation}\label{nuii}
a_2=0,\ \ a_1+a_3\geq 0,\ \ a_i>0\ \text{ for } i\not=1,2,3.
\end{equation}

Set $\ \ \lambda':=\nu+\rho-a_1\alpha_2,\ \ \ \mu:=\lambda-\lambda'$.

One has $(\lambda',\alpha_i)=0$ for $i=1,2$ and
 $(\lambda',\alpha_i)\geq 0$
for $i=0,\ldots,n$.

Write $\lambda-\rho-\nu=\sum_{i=0}^n k_i\alpha_i$. Then $k_i\geq 0$
for each $i$ (since $v\in L(\lambda-\rho)$). Since
$a_1=(\lambda,\alpha_1)+k_0-k_2$, one has $k_2+a_1>0$.
Therefore
$$\mu\in\mathbb{Z}_{\geq 0}\Sigma.$$

By~\Prop{lemCasimir},  $(\nu+2\rho_x,\rho_x)=
||\lambda||^2-||\rho^2||$.
One readily sees that $2(\rho-\rho_x)=(n-2)\alpha_2$,
so $||\rho||^2=||\rho_x||^2$ and
$||\nu+\rho_x||^2=||\nu+\rho||^2$.

 This gives
$||\lambda'||^2=||\lambda||^2$, that is
$$(\lambda,\mu)+(\lambda',\mu)=0.$$
Since $(\lambda,\alpha_i)>0$ and $(\lambda',\alpha_i)\geq 0$ for each $i=0,\ldots,n$,
we obtain $\lambda=\lambda'$. However, $(\lambda',\alpha_2)=0$,
a contradiction.
\end{proof}

\subsection{}
\begin{prop}{propMx2}
Let $N$ be an integrable quotient of an atypical Verma module $M(\lambda)$.

(i)  $F_x(N)\cong L_{\fg_x}(\lambda|_{\fh_x})^{\oplus s}$, where $s=1$ if $N=L(\lambda)$
 and $s=0$ or $s=2$ otherwise.

(ii) Let $(\lambda,\beta)=0$ for an isotropic simple root $\beta$. Then
$$F_x(N)\cong L_{\fg_x}(\lambda|_{\fh_x})^{\oplus s}\text{ where }
\left\{\begin{array}{ll}
s=1& \text{ if }N=L(\lambda),\\
s=0& \text{ if }x\in\fg_{-\beta},\ N\not=L(\lambda),\\
s=2& \text{ if }x\in\fg_{\beta},\ N\not=L(\lambda).
\end{array}\right.$$
\end{prop}
\begin{proof}
By~\ref{maxint},
$M(\lambda)=M_{\Sigma'}(\lambda')$, where $(\lambda',\alpha)=0$ for some isotropic
$\alpha\in\Sigma'$. Thus for (i) we
can assume that $(\lambda,\beta)=0$ for an isotropic simple root $\beta$.
By above, we have $F_x(N)=F_y(N)$, where
$y$ in $\fg_{\beta}$ or
in $\fg_{-\beta}$. Therefore (i) is reduced to (ii). Let us prove (ii).
Clearly, $F_x(N)$ is $\fg_x$-integrable, so completely reducible.
Assume that $Ker_x N$ contains a vector $v$ of weight $\lambda-\mu$
whose image in $F_x(N)$ is a $\fg_x$-singular vector. Since $v\in Ker_x N$
and $v\not\in xN$, one has $(\lambda-\mu,\beta)=0$, that is $(\mu,\beta)=0$.
Since $\mu\in\mathbb{Z}_{\geq 0} \Sigma$, we obtain
$\mu\in\mathbb{Z}_{\geq 0} \Sigma_x+\mathbb{Z}\beta$.

Using~\Lem{lemCasimir} we get
$||\lambda+\rho-\mu||^2=||\lambda+\rho||^2$, that is
$(\lambda+\rho,\mu)+(\lambda+\rho-\mu,\mu)=0$.

Since $N$ is integrable and $(\lambda,\beta)=0$, we get $(\lambda,\alpha)\geq 0$
for each $\alpha\in\Sigma$. Thus $(\lambda+\rho,\mu)\geq 0$ and so
$(\lambda+\rho-\mu,\mu)\leq 0$.

Taking into account that $F_x(N)$ is $\fg_x$-integrable (where $\fg_x=\fsl_n^{(1)}$)
and $\mu\in\mathbb{Z}_{\geq 0} \Sigma_x+\mathbb{Z}\beta$,
$(\lambda+\rho-\mu,\mu)\geq 0$
and  the equality holds if and only if $\mu\in\mathbb{Z}\beta$.
Therefore $\mu\in\mathbb{Z}\beta$, that is $\mu\in \{0,\beta\}$. Hence
$$F_x(N)=L_{\fg_x}(\lambda|_{\fh_x})^{\oplus s},\ \text{ where }
s:=\dim F_x(N_{\lambda}\oplus N_{\lambda-\beta}).$$
Note that $N':=N_{\lambda}\oplus N_{\lambda-\beta}$ is a module over
a copy of $\fsl(1|1)$ generated by $\fg_{\pm\beta}$
(one has $x\in\fsl(1|1)$).
If $N=L(\lambda)$, then $N'$ is a trivial $\fsl(1|1)$-module;
and if $N/L(\lambda-\beta)=L(\lambda)$ then
$N'$ is a Verma $\fsl(1|1)$-module
of highest weight zero. The assertion follows.
\end{proof}

\subsection{}
\begin{cor}{corFx}
Let $L\in\cF_k$ be an irreducible module. Then $F_x(L)=0$ if and only if $L$ is typical.
For atypical $L$, $F_x(L)$ is integrable $\fsl_{n-1}^{(1)}$-module and
$F_x(L)\cong F_x(L')$ if and only $L$ and $L'$ lie in the same block.
\end{cor}
\begin{proof}
Retain notation of~\Cor{extquiver}.
If $L^j,L^{j+1}$ are simple objects in an atypical block $\cB$ and $j\geq 0$ (resp. $j<-1$), then there exists
a Verma module $M(\lambda)$ such that its maximal integrable quotient $V(\lambda)$
such that $V(\lambda)/L^j\cong L^{j+1}$ (resp., $V(\lambda)/L^{j+1}\cong L^j$).
From~\Prop{propMx2}, we get $F_x(L^j)\cong F_x(L^{j+1})$, so $F_x(L)$ is
a non-zero invariant of an atypical block.

Let us show that this invariant separates blocks.
Fix a set of simple roots $\Sigma$ and  take $x\in \fg_{-\alpha_2}$. Let  $\lambda^{\#}\in \fh_x$ be the highest weight of $F_x(L), F_x(L')$. Let us show that $L,L'$ are in the same block. Indeed, each block contains a unique
$\Sigma$-singular irreducible module. Thus we can (and will) assume that $L,L'$
are $\Sigma$-singular. Let $L=L(\lambda), L'=L(\lambda')$.
One has $\lambda^{\#}=\lambda|_{\fh_x}=\lambda'|_{\fh_x}$.
Since $\lambda,\lambda'$ are $\Sigma$-singular, $\lambda=\lambda'$, that is $L\cong L'$ as required.
\end{proof}

\subsection{}
Let us calculate the highest weight of $F_x(L)$.

Let $L=L_{\Sigma}(\lambda)$ be an atypical integrable module of level $k$.
Write $\Sigma=\{\alpha_0\}\cup\dot{\Sigma}$, where $\alpha_0$ is even and
$\dot{\Sigma}$ is a set of simple roots for $\fsl(1|n)$.
Let $\{\vareps_i\}_{i=1}^n\cup\{\delta_1\}$ be the standard notation for $\fsl(1|n)$;
then
$$\alpha_0=\delta-\vareps_1+\vareps_n,\ \alpha_1=\varesp_1-\delta_1, \alpha_2=\delta_1-\vareps_2,\ \ldots,\alpha_n=\vareps_{n-1}-\vareps_n.$$
Set $c_i:=(\lambda+\rho,\vareps_i)$ for $i=1,\ldots,n$ and $d:=(\lambda,\delta_1)$.
Note that these numbers determine  $L$ as a module over $[\fg,\fg]$.

We claim that either $c_1=c_2=b$ and $c_i-b$ is not divisible by $k+n-1$;
there exist a unique index $i$ such that $c_i-b$ is divisible by $k+n-1$.
One has
$$F_x(L(\lambda))=L_{\fsl_{n-1}^{(1)}}(\lambda^{\#}),$$
where $\lambda^{\#}$ has level $k$ and the marks $(\lambda^{\#}+\rho,\vareps_i)$
are obtained from $(c_1,\ldots,c_n)$ by throwing away one element $j$ with
$c_j-b$ divisible by $k+n-1$.

Indeed, set $L=L(\lambda)$. By~\Prop{propMx2}, $F_x(L(\lambda))=L_{\fg_x}(\lambda^{\#})$
for some $\lambda^{\#}\in\fh_x$ (and $\fg_x\cong \fsl_{n-1}^{(1)}$).
There exists $\Sigma'$ such that $L$ is $\Sigma'$-tame and $\Sigma'$
is obtained from $\Sigma$ by $L$-typical odd reflections, so
$\rho wt_{\Sigma} L=\rho wt_{\Sigma'} L$. Let $\beta\in \Sigma'$ be such that
$(\rho wt_{\Sigma'} L,\beta)=0$. Take $y\in\fg_{\beta}$.
By above, $F_x(L)$ is equivalent to $F_y(L)$, where $y\in\fg_{\beta}$
or $y\in\fg_{-\beta}$. Using~\Prop{propMx2} we get
$$F_y(L)=L_{\fg_y}(\lambda'|_{\fh_y}),$$
where $\lambda'=\lambda+\rho-\rho'$ and
$\fh_y=\{h\in \fh\cap \fsl_n^{(1)}| \beta(h)=0\}$.

Assume that  $\lambda+\rho$ is regular. Then there exists a unique $j$ such that
$c_j-b$ is divisible by $k+n-1$.
By above, $\fg_y$ has a set of simple roots $\Sigma_y=\{\alpha\in\Sigma_0|\ (\beta,\alpha)=0\}$. From~\ref{Sigmax} it follows that
for each $\alpha\in\Sigma_y$ one has
$$(\lambda^{\#}+1,\alpha)=(\lambda'+\rho',\alpha)=(\lambda+\rho,\alpha).$$
Therefore $\lambda^{\#}+\rho^{\#}$ has the marks $\{c_i\}_{i=1}^n\setminus\{c_j\}$ as required.

Assume that $\lambda+\rho$ is singular. Then, by~\ref{corLambdak},
$(\lambda+\rho,\alpha_1)=(\lambda+\rho,\alpha_2)=0$ (in particular, $\Sigma'=\Sigma$)
and $\alpha_1+\alpha_2$ is the only even positive root
orthogonal to $\lambda+\rho$.  Then
$c_1=c_2=b$ and $c_i-b$ is divisible by $k+n-1$ if and only if $i=1,2$.
Then $x\in \fg_{\alpha_j}$ for $j=1$ or $j=2$ and, as above,
$\lambda^{\#}+\rho^{\#}$ corresponds to $\{c_i\}_{i=1}^n\setminus\{c_j\}$.

 \section{Modules over simple  affine vertex  superalgebras}\label{sect6}
 Let $\fg$ be an untwisted affine Lie superalgebra, i.e.
 the affinization of a finite-dimensional
 Kac-Moody  Lie superalgebra $\dot{\fg}$.  Let $d\in\fh$
be the standard element ($[d,xt^s]=xt^s$ for $x\in\dot{\fg}$).
One has $\fg=[\fg,\fg]\oplus\mathbb{C}d$.

Recall that a $[\fg,\fg]$-module  (resp., $\fg$-module) $N$ is called {\em restricted} if for every $a\in\dot{\fg}, v\in N$ there exists $n$ such that $(at^m)v=0$
for each $m>n$. A particular case of the restricted $\fg$-modules are the {\em bounded modules}, i.e. the modules where $d$ acts diagonally with integral eigenvalues
bounded from above; as before, denote by $N^{top}$ the eigenspace with the maximal eigenvalue. A bounded module $N$ is called {\em almost irreducible} if any nontrivial
submodule of $N$ has a non-zero intersection with $N^{top}$.

Let $N$ be a restricted $[\fg,\fg]$-module of level $k$
such that $k\not=-h^{\vee}$. The Sugawara construction
equips $N$ with an action of the Virasoro algebra $\{L_n\}_{n\in\mathbb{Z}}$, see~\cite{K3}, 12.8 for details.
Moreover, the $[\fg,\fg]$-module structure on $N$ can be extended to a $\fg$-module structure by setting $d|_N:=-L_0|_N$.

For a restricted $\fg$-module the action of $L_0$ and the Casimir element $\Omega$
are related by the formula $\Omega=2(K+h^{\vee})(d+L_0)$. Therefore
the above procedure assigns to a restricted $[\fg,\fg]$-module of level $k\not=-h^{\vee}$ a restricted $\fg$-module with the zero action of the Casimir operator.

In this section $k\not=-h^{\vee}$ and we identify the restricted $[\fg,\fg]$-modules
of level $k$ and the restricted  $\fg$-modules of level $k$ which are annihilated by the Casimir operator of $\fg$.

Let $\fg^{\#}$ be the "largest affine subalgebra"
 of $\fg_{\ol{0}}$ (see~\S~\ref{largest})
 Let $k$ be such that
 $V_k(\fg)$ is integrable as a $\fg^{\#}$-module (for an appropriate
 normalization of the bilinear form, this means that $k$ is a non-negative integer).
In~\Thm{thmvertex} we prove that the $V_k(\fg)$-modules  are the restricted $[\fg,\fg]$-modules of level $k$
which are $\fg^{\#}$-integrable. We show that
the irreducible bounded $V_k(\fg)$-modules are highest weight modules
if and only if the Dynkin diagram of $\fg_{\ol{0}}$ is connected ($\fg$
is  $\fsl(1|n)^{(1)}$ or
$\mathfrak{osp}(n|m)^{(1)}$ for $n=1,2$), see~\S~\ref{claimpos} and~\ref{badexample2}.
 For $\fg=\fsl(1|n)^{(1)}$ one has
 $\fg^{\#}=\fsl_n^{(1)}$. Hence for a positive integer $k$ the category of
 finitely generated $V_k(\fg)$-modules with finite-dimensional
 weight spaces is the full subcategory of $\cF^k$ with the objects annihilated
 by the Casimir operator.

\subsection{Restricted and bounded modules over affine Lie superalgebras}
If $\fg$ is an affine Lie algebra, then,
by~\cite{GK1}, Thm. 3.2.1, the
 restricted integrable $[\fg,\fg]$-modules are completely reducible and
 the irreducible ones are  highest weight modules. The situation is similar for
$\fg=\mathfrak{osp}(1|2n)^{(1)}$, but is different for other affine Lie superalgebras, see~\S~\ref{exarest} below.

 A bounded  irreducible integrable $[\dot{\fg}_{\ol{0}},\dot{\fg}_{\ol{0}}]^{(1)}$-module is a highest weight module, see~\S~\ref{claimpos}.
An interesting question is whether an {\em irreducible  restricted } $\fg$-module which is $[\dot{\fg}_{\ol{0}},\dot{\fg}_{\ol{0}}]^{(1)}$-integrable is a highest weight module.

\subsubsection{}
\begin{lem}{claimpos}
If $N$ is a bounded $\fg$-module which is $[\dot{\fg}_{\ol{0}},\dot{\fg}_{\ol{0}}]^{(1)}$-integrable, then
$N^{\fn}\not=0$.
\end{lem}
 \begin{proof}
 Set $E:=N^{top}$. Since $E^{\dot{\fn}}\subset N^{\fn}$,
it  is enough to show that $E^{\dot{\fn}}\not=0$.

Set $\ft:=[\dot{\fg}_{\ol{0}},\dot{\fg}_{\ol{0}}]$.
Note that $E$ is a  $\dot{\fg}$-module which is $\ft$-integrable. Therefore
$E$ is a direct sum of finite-dimensional $\ft$-modules. In particular,
$\dot{\fn}_{\ol{0}}$ acts locally nilpotently on $E$. Therefore
$\dot{\fn}$ acts locally nilpotently on $E$. Let $0=\ft^{0}\subset\ft^{1}\subset\ldots\subset\ft^s=\dot{\fn}$
be the derived series of $\dot{\fn}$ ($\ft^{i}=[\ft^{i+1},\ft^{i+1}]$).
Set $E(0):=E$ and $E(i):=E(i-1)^{\ft^{i}}$ for $i=1,\ldots, s$.
By induction  $E(i)\not=0$, since
$\ft^{i}/\ft^{i-1}$ is a finite-dimensional abelian Lie superalgebra which acts
locally nilpotently on $E(i-1)$. Hence $E^{\dot{\fn}}=E(s)\not=0$
as required.
\end{proof}

\subsubsection{}\label{exarest}
 By above, an irreducible  bounded $\fsl(1|n)^{(1)}$-module which is $\fsl_n^{(1)}$-integrable is $\fsl(1|n)^{(1)}$-integrable (since it is a highest weight module).
Below we  give an example of a cyclic bounded
$\fsl(1|2)^{(1)}$-module which is $\fsl_2^{(1)}$-integrable, but is
not $\fsl(1|2)^{(1)}$-integrable (the action of $\fh$ is not locally finite).

Consider the usual $\mathbb{Z}$-grading on $\fsl(1|2)$:
$\dot{\fg}=\dot{\fg}_{-1}\oplus \dot{\fg}_{0}\oplus \dot{\fg}_1$, where $\dot{\fg}_0=\dot{\fg}_{\ol{0}}=\fsl_2\times \mathbb{C}z$
and $\dot{\fg}_{\pm 1}$ are irreducible $\fsl_2$-modules.
Let $f,h,e$ be the standard generators of $\fsl_2$. Consider the triangular
decomposition of $\fg$ with $\fn=\mathbb{C}e+\dot{\fg}_1+\sum_{s=1}^{\infty}\dot{\fg}t^s$.

View $\mathbb{C}[z]$ as a module over $\fp:=\fh+\dot{\fg}_0+\fn$ by the trivial action
of  $\mathbb{C}d+\dot{\fg}_0+\fn$
and $K$ acting by $Id$. Consider the induced module $M:=Ind_{\fp}^{\fg} \mathbb{C}[z]$.
Then $M$ has the central charge $1$
and $M^{top}$ is a free $\mathbb{C}[z]$-module. As an $\fsl_2$-module
$M^{top}$ is the direct sum of countably many
copies of $\Lambda\fg_{-1}$, so
$e^2 M^{top}=0$.

From~\Cor{corV} below it follows that
$M$ has an almost irreducible quotient $N$ which is
  $\fsl_2^{(1)}$-integrable and $N^{top}=M^{top}$. Since $z$ acts freely on $M^{top}$,
  $N$ is not $\fsl(1|n)^{(1)}$-integrable. Note that $M$
  is bounded and cyclic (generated by the image of $1\in\mathbb{C}[z]$), so $N$ is also bounded and cyclic. It is not hard to see that the Casimir acts freely on $N$.

  \subsection{The subalgebra $\fg^{\#}$}\label{largest}
 Recall that (for affine $\fg$) the $[\dot{\fg}_{\ol{0}},\dot{\fg}_{\ol{0}}]^{(1)}$-integrable modules
 exist only on level zero or in the case when the Dynkin diagram of
 $\fg_{\ol{0}}$ is connected, see~\cite{KW2}.
 We consider the integrability with respect to the "largest affine subalgebra"
 of $\fg_{\ol{0}}$, see below.

 Recall that $\dot{\fg}_{\ol{0}}$ is a  reductive Lie algebra
 and it can be decomposed as $\dot{\fg}_{\ol{0}}=\dot{\fg}^{\#}\times \dot{\ft}$,
 where $\dot{\fg}^{\#}$ is a simple Lie algebra (the "largest part" of $\dot{\fg}_{\ol{0}}$), i.e.:

 for $\dot{\fg}=\fsl(m|n), \mathfrak{osp}(m|n)$ with $n\geq m$
 one has $\dot{\fg}^{\#}=\fsl_n, \mathfrak{sp}_n$ respectively;

 for $\dot{\fg}=\mathfrak{osp}(m|n)$ with $m>n$ one has $\dot{\fg}^{\#}= \mathfrak{so}_m$;

 for the exceptional Lie superalgebras $F(4), G(3)$  one has $\dot{\fg}^{\#}=B_3, G_2$ respectively;

 for $D(2,1,a)$ we have $\dot{\fg}_{\ol{0}}=A_1\times A_1\times A_1$ with the corresponding roots $\alpha_1,\alpha_2,\alpha_3$ subject to the relation
 $||\alpha_1||^2:||\alpha_2||^2:||\alpha_3||^2=1:a:(-a-1)$;
 we take $\dot{\fg}^{\#}=A_1$, which corresponds any copy of $A_1$ if $a\not\in\mathbb{Q}$ and the first copy of $-1<a<0$ (see~\cite{GK2}, 6.1).

We have a natural embedding of the affine algebra $\fg^{\#}$
(which is the affinization of $\dot{\fg}^{\#}$) to $\fg_{\ol{0}}$.

\subsection{Modules over affine  vertex superalgebras}
Let $V^k(\fg)$ be the
 affine vertex algebra and $V_k(\fg)$ be its simple quotient.

There is a natural equivalence between
the categories of $V^k(\fg)$-modules and the restricted $[\fg,\fg]$-modules of level $k$ if $k$ is not critical, see~\cite{FZ}, Thm. 2.4.3.

If $\fg$ is a Lie algebra and $k\not=0$ is such that $V_k(\fg)$ is $[\fg,\fg]$-integrable, then the $V_k(\fg)$-modules correspond to the integrable $[\fg,\fg]$-modules, see~\cite{FZ}, Thm. 3.1.3 and~\cite{DLM}, Thm. 3.7.

\subsubsection{}
\begin{thm}{thmvertex}
If $V_k(\fg)$ is integrable as a $\fg^{\#}$-module, then the $V_k(\fg)$-modules
are the restricted $[\fg,\fg]$-modules of level $k$ which are integrable over $\fg^{\#}$.
As $\fg^{\#}$-modules
these modules are direct sums of irreducible integrable highest weight modules.

The $V_0(\fg)$-modules are the direct sums of the trivial modules.
\end{thm}

\begin{rem}{remy}
Normalize the  non-degenerate bilinear form by the condition $(\alpha,\alpha)=2$,
where $\alpha$ is the longest root in $\dot{\fg}^{\#}$. Then $V_k(\fg)$ is integrable over $\fg^{\#}$ if and only if $k$ is a non-negative integer.
\end{rem}

 \subsection{Proof of~\Thm{thmvertex}}
Introduce the vacuum $\fg$-module of level $k$:
$$V^k:=\Ind_{\dot{\fg}+\fn+\fh}^{\fg} \mathbb{C}_k,$$
where $\mathbb{C}_k$ is the trivial $\dot{\fg}+\fn$-module with
$K$ acting by $k Id$ and $d$ acting by zero. As a $[\fg,\fg]$-module
$V^k(\fg)$ is isomorphic to $V^k$.

Let $\Lambda_0\in\fh^*$ be such that $\Lambda_0(K)=1, \Lambda_0(\dot{\fh})
=\Lambda_0(d)=0$. Note that $V^k$ is a $\dot{\fg}_{\ol 0}$-integrable quotient of
the Verma module $M(k\Lambda_0)$ and that $L(k\Lambda_0)$ is a unique simple quotient of $V^k$.
As a $[\fg,\fg]$-module
$V_k(\fg)$ is isomorphic to $L(k\Lambda_0)$.

Recall that for a given non-degenerate bilinear form $h^{\vee}=(\rho,\delta)$,
where $\rho$ is the Weyl vector and $\delta$ is the minimal imaginary root.

\subsubsection{}
\begin{thm}{thmvert1}
Let $k\not=-h^{\vee}$ be such that $L(k\Lambda_0)$ is $\fg^{\#}$-integrable.
Then $L(k\Lambda_0)$ is a unique $\fg^{\#}$-integrable quotient
of $V^k$.
\end{thm}
\begin{proof}
From~\cite{GK2} it follows that
the character of a $(\fg^{\#}+\dot{\fg}_{\ol 0})$-integrable quotients of a non-critical
Verma module $M(k\Lambda_0)$ is given by the KW-character formula
(see Section 4  for the cases $h^{\vee}\not=0$ and for $A(n,n)^{(1)}$, and
 Section 6 for the remaining cases). Since $V^k$ is $\dot{\fg}_{\ol 0}$-integrable,
 its $\fg^{\#}$-integrable quotient is $(\fg^{\#}+\dot{\fg}_{\ol 0})$-integrable.
Hence such  quotients have the same character formula, so such quotient is unique.
\end{proof}

\subsubsection{}
We denote by $\vac$ the highest weight vector of $V^k$ (and its image in $L(k\Lambda_0)$).

We normalize the bilinear form as in~\Rem{remy} and fix a triangular decomposition in $\dot{\fg}$ in such a way
that the maximal root $\theta$ lies in the root system of $\dot{\fg}^{\#}$.
Then $\alpha_0=\delta-\theta$ is a simple root and
 $(\alpha_0,\alpha_0)=2$. Let $f_0$ be a non-zero element in $\fg_{-\alpha_0}$;
 note that $f_0\in\fg^{\#}$.

\subsubsection{}
\begin{cor}{corvert1}
Let  $k\not=-h^{\vee}$ be such that
$L(k\Lambda_0)$ is  $\fg^{\#}$-integrable. Then $L(k\Lambda_0)=V^k/I$,
where the submodule $I$ is generated by  $f_0^{k+1}\vac$.
\end{cor}
\begin{proof}
Since $L(k\Lambda_0)$ is $\fg^{\#}$-integrable,
$f_0^{k+1}\vac$ is a singular vector in $V^k$.
Let $I$ be the submodule of $V^k$ generated by this vector.
By~\Thm{thmvert1}, it is enough to show that $V^k/I$ is $\fg^{\#}$-integrable.
From~\cite{K3},
Lemmas 3.4, 3.5, it suffices to check that
for each $\alpha$ in the set of simple roots of $\fg^{\#}$
the root spaces $\fg_{\pm\alpha}$ act nilpotently on $v$, where
 $v$ is the image of $\vac$ in $V^k/I$.
Clearly, $\fg_{\pm\alpha}\vac=0$  for $\alpha\not=\alpha_0$
and $\fg_{\alpha_0}v=\fg_{-\alpha_0}^{k+1}v=0$. The assertion follows.
\end{proof}

 \subsubsection{}
\begin{rem}{remy2}
\Thm{thmvert1} and~\Cor{corvert1} hold also in the case when $\fg$ is a twisted
 affinization ($\fg$ is any symmetrizable affine Lie superlagebra).
 In~\Cor{corvert1} the following change should be done if
$\frac{\alpha_0}{2}\in\Delta$: $f_0$ should be chosen in $\fg_{-\alpha_0/2}$
and $I$ is generated by  $f_0^{2k+1}\vac$. The proofs are the same.
\end{rem}

 \subsubsection{}
For each $a\in V^k(\fg)$ let $Y(a,z)$ be the corresponding vertex operator.
 The following lemma has to be standard (see, for example,~\cite{AM}, Prop. 3.4).

\begin{lem}{lemvert1}
Let $I\subset V^k(\fg)$ be a cyclic submodule generated by a  vector
$a\in V^k(\fg)$.
A $V^k(\fg)$-module $N$ is a $V^k(\fg)/I$-module if and only if
$Y(a,z) N=0$.
\end{lem}

\subsubsection{}
By~\cite{GK1}, Thm. 3.2.1 any
 restricted integrable $[\fg^{\#},\fg^{\#}]$-modules
 is  a direct sum of irreducible integrable highest weight
$[\fg^{\#},\fg^{\#}]$-modules. Let us show that $V_k(\fg)$-modules
are restricted $[\fg,\fg]$-modules of level $k$ which are integrable over $[\fg^{\#},\fg^{\#}]$.

Take $k=0$. Then $V_k(\fg)$ is one-dimensional and $V_k(\fg)$-modules are
restricted $[\fg,\fg]$-modules of zero level which are annihilated by $[\fg,\fg]$.
Hence these are modules are the direct sums of trivial modules.

Take $k\not=0$. Then  $k\not=-h^{\vee}$.
From~\Lem{lemvert1} and~\Cor{corvert1}, we conclude that $V_k(\fg)$-modules are restricted $[\fg,\fg]$-modules
which are annihilated by $Y(f_0^{k+1}\vac,z)$.
Note that $Y(f_0^{k+1}\vac,z)\in V^k(\fg^{\#})$ and
$V_k(\fg^{\#}):=V^k(\fg^{\#})/I'$, where $I'$ is the $\fg^{\#}$-submodule
of $V^k(\fg^{\#})$ which is
generated by $f_0^{k+1}\vac$. Therefore the modules over $V_k(\fg)$
are those, which, viewed as $V^k(\fg^{\#})$-modules, are the modules over $V_k(\fg^{\#})$.
By~\cite{DLM}, Thm. 3.7, the  $V_k(\fg^{\#})$-modules
are  direct sums of irreducible integrable highest weight
$[\fg^{\#},\fg^{\#}]$-modules of level $k$.
We conclude that the  $V_k(\fg)$-modules
are the restricted integrable $[\fg^{\#},\fg^{\#}]$-modules of level $k$ as required.
This completes the proof of~\Thm{thmvertex}
\qed

\subsection{Positive energy $V_k(\fg)$-modules}
A  $V^k(\fg)$-module is called  {\em positive energy}
(see~\cite{DK}) if it is
$\mathbb{Z}$-graded $[\fg,\fg]$-modules of level $k$:
$M=\oplus_{m\in \mathbb{Z}} M_m$
with $(at^n)M_m\subset M_{m-n}$ with the grading bounded from the below.
For such a module we extend the $[\fg,\fg]$-action to the $\fg$-action by
 $dv:=-mv$ for $v\in M_m$. Thus
 the positive energy  $V^k(\fg)$-modules correspond to  the bounded $\fg$-modules of level $k$. (In~\cite{DLM} a similar object is called an admissible module;
in~\cite{FZ} all modules are assumed to be of this form.)

 A positive energy $V^k(\fg)$-module  is {\em ordinary} (see~\cite{DLM})
 if the grading is given by the action of $L_0$ and the homogeneous components are finite-dimensional. Thus the ordinary modules are the bounded $\fg$-modules of level $k$ with the zero action of the Casimir operator.

 \subsubsection{}
Let $\fg$ be such that the Dynkin diagram of $\fg_{0}$ is connected.
As for $\fsl(1|n)^{(1)}$-case, we call a $\fg$-module integrable if it is integrable over $[\dot{\fg}_{\ol{0}},\dot{\fg}_{\ol{0}}]^{(1)}$ and $\fh$ acts locally finite
with finite-dimensional generalized $\fh$-eigenspaces.
\Lem{claimpos} gives the following

\begin{cor}{corpos}
Let $\fg$ be such that the Dynkin diagram of $\fg_{0}$ is connected
and $k$ be a non-negative integer.
Then a positive energy $V_k(\fg)$-module
contains a singular vector ($v$ such that $\fn v=0$).
In particular, the irreducible positive energy $V_k(\fg)$-module
are $\fg$-integrable highest weight modules of level $k$.
\end{cor}

\subsubsection{}
Let $V$ be VOA and $A(V)$ be its Zhu algebra.

Thm. 2.30 in~\cite{DK} (see also Thm. 2.2.1 in~\cite{Z})
for the trivial twisting gives

\begin{prop}{}
The restriction functor $N\mapsto N^{top}$ is a functor
from the category of positive energy $V$-modules up to a shift of grading
to the category of $A(V)$-modules,
which is inverse to the induction functor
$E\mapsto V(E)$ from the $A(V)$-modules
to the full subcategory
of almost irreducible $V$-modules (up to a shift of grading).
In particular, these functors establish   a
bijective correspondence between the irreducible positive energy
$V$-modules and the irreducible $A(V)$-modules.
\end{prop}

\subsubsection{}
As in~\cite{FZ} Thm. 3.1.1, 3.1.2, the Zhu algebra of $V^k(\fg)$ is $\, \cU(\dot{\fg})$ and
the Zhu algebra of $V_k(\fg)$ is $\cU(\fg)/(e_{\theta}^{k+1})$, where
$f_0=e_{\theta}t^{-1}$ ($e_{\theta}\in\dot{\fg}_{\theta}$).

\begin{cor}{corV}
Let $k$ be a non-negative integer and
let $E$ be a $\dot{\fg}$-module satisfying $e_{\theta}^{k+1}E=0$.
 There exists  a unique almost irreducible $\fg^{\#}$-integrable $\fg$-module
$N={\oplus}_{i=0}^{\infty} N^i$ of level $k$
such that $N^i$ is the $i$th eigenspace of $-d$ and
$N^0=E$. This module has a natural structure of positive energy $V_k$-module. Moreover,
$N$ is irreducible  if and only if $E$ is irreducible.
\end{cor}

\subsubsection{}\label{badexample2}
Let $k$ be a non-negative integer.

If the Dynkin diagram of $\fg_{\ol{0}}$ is connected, then, by~\Cor{corpos},
an  irreducible $V_k(\fg)$-module is a highest weight modules.
Let us show that this does not hold if
the Dynkin diagram of $\fg_{\ol{0}}$ is not connected.

Let $\fg$ be such that the Dynkin diagram of $\fg_{\ol{0}}$ is not connected.
In this case $\dot{\fg}_{\ol{0}}=\dot{\fg}^{\#}\times \ft$, where $\ft$ is semisimple.
Take any irreducible $\ft$-module $E$ and view it as $\dot{\fg}_{\ol{0}}$-module
via the trivial action of $\dot{\fg}^{\#}$. Set $E':=\Ind_{\dot{\fg}_{\ol{0}}}^{\dot{\fg}} E$. As $\dot{\fg}^{\#}$ is the direct sum of copies of $\Lambda\dot{\fg}_{\ol{1}}$, so there exists
$m$ such that $e_{\theta}^m E'=0$.  Let $L$ be an irreducible quotient of $E'$.
By~\Cor{corV}, for each integral $k\geq m-1$
 there exists  an irreducible bounded $\fg^{\#}$-integrable $\fg$-module $N$ of level $k$ such that $N^{top}=L$. Note that $E$ is an $\ft$-quotient of $L$.
 In particular, $\fh$ acts locally finitely on $N$ if and only if
the Cartan algebra of $\ft$ acts locally finitely on $E$.

\subsubsection{}
Take $\fg=\fsl(2|1)^{(1)}$.
The example in~\S~\ref{exarest} gives a cyclic $\fg$-bounded $\fg^{\#}$-integrable module of level $1$ with a free action of the Casimir operator.
Viewing this module as a $[\fg,\fg]$-module we obtain a cyclic almost irreducible positive energy $V_1(\fg)$-module
with a free action of $L_0$ (in particular, this module is not ordinary
and is not $\fg$-integrable).

\end{document}